\documentclass{amsart}
\usepackage{amssymb,latexsym, amsmath, amscd, array, graphicx}

\usepackage{graphicx}
\usepackage{tikz-cd}

\swapnumbers
\numberwithin{equation}{section}

\theoremstyle{plain}

\newtheorem{thm}{Theorem}[section]
\newtheorem{cor}[thm]{Corollary}

\newtheorem{question}[thm]{Problem}
\newtheorem{prop}[thm]{Proposition}

\theoremstyle{definition}
\newtheorem{defin}[thm]{Definition}

\newtheorem{rem}[thm]{Remark}
\newtheorem{remark[thm]}{Remark}

\newtheorem{ex}[thm]{Example}

\def\Ker{\protect\operatorname{Ker}}

\def\Wi{\widetilde}

\def\Z{{\mathbb Z}}
\def\Q{{\mathbb Q}}
\def\R{{\mathbb R}}

\def\1{\hbox{\rm\rlap {1}\hskip.03in{\rom I}}}
\def\Bbbone{{\rm1\mathchoice{\kern-0.25em}{\kern-0.25em}
{\kern-0.2em}{\kern-0.2em}I}}

\long\def\forget#1\forgotten{} %

\begin{document}

\title[On Lipschitz cohomology of aspherical manifolds ]
{On Lipschitz cohomology of aspherical manifolds}
\author[A.~Dranishnikov]
{Alexander Dranishnikov}

\thanks{The author was supported by Simons Foundation}

\address{A. Dranishnikov, Department of Mathematics, University
of Florida, 358 Little Hall, Gainesville, FL 32611-8105, USA}
\email{dranish@math.ufl.edu}



\begin{abstract} 
We introduce the notion of Lipschitz cohomology classes of a group with locall coefficients and reduce the Novikov conjecture for a gropup $\Gamma$
 to the question whether the Berstein-Schwarz class $\beta_\Gamma\in H^1(\Gamma,I(\Gamma))$ is Lipschitz.
\end{abstract}

\maketitle

\section{Introduction}
The Novikov Higher Signature conjecture states that the higher signatures of closed oriented
manifolds are invariant under orientation-preserving homotopy equivalences. For aspherical manifolds it is equivalent to say that
the rational Pontryagin classes are homotopy invariant. The conjecture has a long and eventful history (see the survey~\cite{Yu}) and it is still open.

One of the approaches to the Novikov conjecture was introduced more than 30 years ago by Connes, Gromov, and Moscovici by defining 
Lipschitz cohomology classes of a discrete group and proving the Novikov conjecture for such classes~\cite{CGM}.
Their approach recovered the Novikov conjecture for all classes of groups where it was proven by that time. In particular, it gave another proof of
the Novikov conjecture for hyperbolic groups. More than 20 years ago T. Kato used Lipschitz cohomology to prove  the Novikov conjecture for
the combable groups ~\cite{K1},\cite{K2}. In~\cite{Dr} Lipschitz cohomology were used to recover Carlson-Pedersen result~\cite{CP} on the Novikov conjecture
for groups $\Gamma$ admitting certain equivariant compactification of the classifying space $E\Gamma$. 

Since then the Lipschitz cohomology approach to the Novikov conjecture seems to be
completely abandoned. This paper is an attempt to revive it.

Lipschitz cohomology classes of a discrete group were defined  in [CGM] as the collection of images under a family of  certain slant product homomorphisms
denoted $\alpha_{\cap}$ and defined by maps $\alpha:\Gamma\times P\to\R^n$ where $P$ is a parameter space. Connes, Gromov, and Moscovici 
first defined a light version of such classes called {\em proper Lipschitz cohomology}.
It turns out that in all cases where all cohomology classes of an aspherical manifold are known to be proper Lipshitz, like in the case of manifolds 
with non-positive curvature $K(M)\le0$,
they all 
belong to the image of one homomorphism $\alpha_{\cap}$. In such situation we say that
the group $\Gamma$ has {\it canonically Lipschitz cohomology}. Unfortunately the canonically Lipschitz cohomology cannot lead to a proof of the Novikov conjecture, since
it was shown in ~\cite{Dr} that there are aspherical manifolds
which do not have all its rational cohomology classes canonically Lipschitz. 

It was noticed in~\cite{CGM} that to prove that the 0-dimensional cohomology class $1\in H^0(\Gamma,\R)=\R$ is proper Lipschitz  could be a difficult problem. 
In the second half of the paper~\cite{CGM} the authors
introduced a heavier version of Lipschitz cohomology classes for which the Novikov conjecture still holds true and for which is quite easy to show that
the class  $1\in H^0(\Gamma,\R)$ is Lipschitz for any group.

In this paper we explore the possibility of proving the Novikov conjecture for a group $\Gamma$ by proving that
its Berstein-Schwarz class $\beta_\Gamma\in H^1(\Gamma,I(\Gamma))$ is Lipschitz. Here $I(\Gamma)$ is the augmentation ideal. 
Due to universality of $\beta_\Gamma$~\cite{DR},\cite{Sc} every cohomology class of $\Gamma$ is the image of some power of 
$\beta_\Gamma$ under a coefficient homomorphism.
In order to use this property of $\beta_\Gamma$  with the aim the Novikov conjecture we extend the notion of Lipschitz 
cohomology classes to cohomology  with local coefficients. Like in~\cite{CGM} we do it in two steps: First we define  proper Lipschitz cohomology 
and then general Lipschitz cohomology.

We note that, since the authors in~\cite{CGM} worked with differential forms, their definition of Lipschitz cohomology classes makes sense only for real coefficients. 
In this paper we use a combinatorial approach. For that, first we give an integral version of Connes-Gromov-Moscovici definition of proper 
Lipschitx cohomology with constant coefficients
and then
we transform it to a definition of proper
Lipschitz cohomology with coefficients in a $\Z\Gamma$-module. We prove the product theorem, the coefficient homomorphism theorem, 
and the connecting homomorphism theorem
for proper Lipschitz cohomology with local coefficients. These results allow us to reduce the Novikov conjecture to the question about proper Lipschitz property of
the 0-dimensional class $1\in H^0(\Gamma,\Z)=\Z$. As it was already noted above, this  is far from being easy.

In~\cite{CGM} the definition of proper Lipschitz cohomology classes of a group $\Gamma$ was extended by dropping at some the properness of 
$\Gamma$-action on the paramiter space $P$. 
We extend our definition to the non-proper case similarly. The bad news about  general Lipschitz cohomology with local coefficients is that the connecting
homomorpphism theorem does not hold anymore.
It means that we cannot reduce the Novikov conjecture to the above question about 0-dimensional cohomology. The good news is that
he product theorem and the coefficient homomorphism theorem still hold true.
In view of universality of $\beta_\Gamma$ we can reduce the Novikov conjecture to
the question about 1-dimensional cohomology. Namely, we reduce it to the question whether the Berstein-Schwarz class $\beta_\Gamma$ is Lipschitz. 

Also in the paper we computed the cohomology group $H^k(\Gamma,I(\Gamma)^{\otimes k})=\Z$ in the case when $\Gamma$ is the fundamental group of a closed aspherical orientable manifold. 
Thus, since the Novikov conjecture is rational, to prove it for a group $\Gamma$ it suffices to find any nonzero Lipschitz cohomology class in $H^1(\Gamma,I(\Gamma))$.

\

The paper is organized as follows. In Section 2 we give a combinatorial definition of proper Lipschitz cohomology with constant coefficients and present  examples of aspherical manifolds with
canonically proper Lipschitz integral cohomology.
In Section 3 we define the Berstein-Schwarz cohomology class, state the Universality Theorem, and carry out computations of cohomology
 of aspherical manifolds with coefficients in $I(\Gamma)^{\otimes k}$. In Section 4 first we give definition of proper and then general Lipschitz cohomology with local coefficients.
There we prove the coefficient homomorphism, the product, and the connecting homomorphism theorems.
\subsection{Some notations} We denote cohomology of a group $\Gamma$ with coefficients in $\Z\Gamma$-module $L$ as $H^n(\Gamma,L)$ and cohomology of a space
$X$ with $\pi_1(X)=\Gamma$ as $H^n(X;L)$. Thus, $H^n(B\Gamma;L)=H^n(\Gamma,L)$. For $\Z\Gamma$-module $L$ and $M$ the tensor product $L\otimes M$ means the tensor product over $\Z$ and $L\otimes_\Gamma M$ means the tensor product over $\Z\Gamma$. The tensor product of $n$ copies of $L$ over $\Z$ will be denoted as $L^{\otimes n}$.
The group of covariants of $L$ is denoted by $L_\Gamma$.

\section{Proper Lipschitz Cohomology of groups}

In~\cite{CGM} Connes-Gromov-Moscovici defined proper Lipschitz cohomology with coefficients in $\R$.
In this section we present a combinatorial definition which works for any ring, in particular for  $\Z$ and $\Q$.

\subsection{Equivariant slant product}
We recall the slant product. Let $C$ and $D$ be chain complexes~\cite{Sp}:
$$/: Hom(C\otimes D,\Z)\times D\to Hom(C,\Z)$$ is defined by the equality $(\phi/c')(c)=\phi(c\otimes c')$. Thus, if $\phi\in Hom((C\otimes D)_n,\Z)$ and $c'\in D_k$, then
$\phi/c'\in Hom(C_{n-k},\Z)$. This slant product defines the slant product for cohomology
$$
/:H^n(C\otimes D)\times H_k(D)\to H^{n-k}(C).
$$

Now we assume that the group $\Gamma$ acts on the chain complexes $C$. Then there is  a $\Gamma$-equivariant version of the slant product:
$$/: Hom_\Gamma(C\otimes D,\Z)\times D\to Hom_\Gamma(C,\Z)$$ 
where $Hom_\Gamma(C,\Z)$ denotes the group of $\Gamma$-invariant homomorphisms.

Suppose that the discrete group $\Gamma$ acts freely on a finite dimensional locally finite simplicial complex $P$ by simplicial transformations. 
Let $C$ be  a free resolution of $\Z$ over $\Z\Gamma$ and
let $D=C(P)^\Gamma$ to be the chain complex that consist of $\Gamma$-invariant infinite simplicial chains on $P$. 
Thus, $D$ is isomorphic to the chain complex of  infinite simplicial chains $C^{inf}(P/\Gamma)$ on the orbit sapce.

We consider the equivariant slant product
$$
H^n_\Gamma(C\otimes D)\times H_k(P:\Gamma) \to H^{n-k}_\Gamma(C)=H^{n-k}(\Gamma)
$$
where $H_k(P:\Gamma)=H_k(C^{inf}(P/\Gamma))$. Note that if $P/\Gamma$ is compact, then $H_k(P:\Gamma)=H_k(\Gamma)$.
\subsection{Proper Lipschitz cohomology classes}
The proper Lipschitz cohomology classes of a discrete group  $\Gamma$ were defined by Connes-Gromov-Moscovici ~\cite{CGM} by means of
the following data:
A finite-dimensional connected locally finite simplicial complex $P$ with a free  simplicial $\Gamma$-action
and a  map $\alpha:
\Gamma\times P\to\R^n$ satisfying the conditions
\begin{itemize}
\item (1) $\alpha$ is invariant with respect to the diagonal action of $\Gamma$ on
$\Gamma\times P$;
\item (2) the restriction $\alpha\mid_{\gamma\times P}$ is proper for all $\gamma\in \Gamma$;
\item (3) the restriction $\alpha\mid_{\Gamma\times p}$ is 1-Lipschitz for all $p\in P$ for the word metric on $\Gamma$.
\end{itemize}

Let $D$ be as above and let $C=C_*(\Delta^\infty(\Gamma))$ be the simplicial chain complex of the infinite simplex spanned by $\Gamma$.
Given the above map $\alpha:\Gamma\times P\to\R^n$ we define an extension $\bar\alpha:\Delta^\infty(\Gamma)\times P\to\R^n$ by means of the barycentric coordinates
$$
\bar\alpha((\sum_{i=0}^kt_i\gamma_i)\times x)=\sum_{i=0}^kt_i\alpha(\gamma_i\times x).
$$
Let $\omega$ be a singular cocycle generating $H^n_c(\R^n)$ with the support in some ball $B\subset\R^n$.
It turns out to be the cochain $\bar\alpha^*(\omega)$ is well-defined on infinite chains $C(\Gamma)\otimes D$. 
Let $\sum_{\sigma}n_\sigma$ be an infinite $\ell$-dimensional simplicial integral chain on $P$ where the sum is taken over all $\ell$-simplices
in $P$. Let $\Delta$ be a simplex in $\Delta^\infty(\Gamma)$ of dimension $k=n-\ell$.
Then by the definition
$$
\bar\alpha^*(\omega)(\Delta\otimes\sum_{\sigma}n_\sigma\sigma)=\sum_{\sigma}n_\sigma\omega(\bar\alpha|(\Delta\times\sigma))
$$
where the restriction of $\bar\alpha$ to $\Delta\times\sigma$ is denoted as $\bar\alpha|(\Delta\times\gamma\sigma)$
and  is treated as the singular $n$-chain on $\R^n$.
Here we use the staircase triangulation~\cite{L} of the product of oriented simplices $\Delta\times\sigma$ to define the simplicial chain.

By virtue of condition (2) the map $\bar\alpha:\Delta\times P\to\R^n$ is proper on $y\times P$ for each vertex $y\in\Delta$ and 
by condition (3)
the image of $\Delta\times x$ is uniformly bounded on $x\in P$. Therefore, the restriction of $\bar\alpha$ to $\Delta\times P$ is proper.
Hence there is a compact subset $D\subset P$ such that $\bar\alpha(\Delta\times(P\setminus D))\cap B=\emptyset$. 
Since $D$ intersects only finitely many simplices $\sigma$, the sum above is finite.
Thus, $\bar\alpha^*(\omega)$ is well-defined. 

If the chain $z=\sum_{\sigma}n_\sigma$ is $\Gamma$-invariant, the value $\bar\alpha^*(\omega)(\Delta\otimes\hat\sigma)$ is $\Gamma$-invariant on the first factor:
$$\bar\alpha^*(\omega)(\gamma'\Delta\otimes z)=\sum_{\hat\sigma\subset P/\Gamma}\sum_{\gamma\in\Gamma}n_\sigma\omega(\bar\alpha|(\gamma'\Delta\times\gamma\sigma))=
$$
$$
\sum_{\hat\sigma\subset P/\Gamma}\sum_{\gamma\in\Gamma}n_{\sigma}\omega(\bar\alpha|(\Delta\times(\gamma')^{-1}\gamma\sigma)=\bar\alpha^*(\omega)(\Delta\otimes z).$$
Here we split the infinite sum $z=\sum_{\sigma}n_\sigma=\sum_{\hat\sigma}\sum_{\gamma\in\Gamma}n_\sigma(\gamma\sigma)$
into the sums along different orbits $\hat\sigma=\Gamma\sigma$ of simplices.

Note that $\bar\alpha^*(\omega)$ is cocycle as the image of a cocycle.
Let $[\bar\alpha^*(\omega)$] denote the cohomology class of the cocycle $\bar\alpha^*(\omega)$.

We define the map $\alpha_\cap:H_k(P:\Gamma)\to H^{n-k}(\Gamma)$ as $$\alpha_\cap(a)=[\bar\alpha^*(\omega)]/a.$$

\begin{defin} All classes in the image $\alpha_{\cap}(H_*( P:\Gamma))\subset
H^*(\Gamma)$ are called integral {\it proper Lipschitz cohomology classes} of a group $\Gamma$.
\end{defin}

\begin{question}
Is the 0-dimensional cohomology class $1\in H^0(\Gamma,\Z)$ proper Lipschitz for every group $\Gamma$?
\end{question}
In the rest of the paper we will try to convince the reader that this is an important question in the Lipschitz cohomology theory (see Corollary~\ref{2}).

\

The above definition can be done in terms of a cellular $\Gamma$-action on $P$ and it can be done directly with use of $\bar\alpha$.
What we need for that is the finiteness condition of the sum $\sum_{\gamma\in\Gamma}\omega(\bar\alpha|(\Delta\times\gamma\sigma))$.
Suppose that a CW complex $B\Gamma$ is given a metric that agrees with its topology on every finite subcomplex. We consider the induced 
metric and CW structure on the universal cover $E\Gamma$.  
\begin{prop}
Let $B\subset\R^n$ be the unit ball centered at 0. Suppose that  a map
$\bar\alpha:E\Gamma\times P\to\R^n$ satisfies the conditions 
\begin{itemize}
\item (1) $\bar\alpha$ is invariant with respect to the diagonal action of $\Gamma$ on
$\Gamma\times P$;
\item (2) the restriction $\bar\alpha\mid_{x\times P}$ is proper for all $x\in E\Gamma$;
\item (3) the restriction $\bar\alpha\mid_{E\Gamma\times y}$ is 1-Lipschitz for all $y\in P$.
\end{itemize}
Then for every cells $e$ in $E\Gamma$ and $c$ in $P$ the intersection $\bar\alpha(e\times\gamma c)\cap B=\emptyset$ for all but finitely many $\gamma\in\Gamma$.
\end{prop}
\begin{proof}
We fix $y_0\in P$. Suppose that $diam(e)\le d$, for $d>1$.
For $x\in e$ in view of the propernes of map $\bar\alpha|_{x\times P}:x\times P\to\R^n$ there is a compact neighborhood $D$ of $y_0$ such that 
$\bar\alpha(x\times(P\setminus D))\subset\R^n\setminus B(0,2d)$. Since by the condition (3), $\bar\alpha(e\times y)$ has diameter $\le d$ and
we obtain that $\bar\alpha(e\times(P\setminus D))\subset\R^n\setminus B(0,d)$. The properness of the action of $\Gamma$ on $P$ implies that only 
finitely many translates of $c$ lie in $D$.
\end{proof}

\begin{ex}
The space of probability measures on $\Gamma$ with finite supports $P_f(\Gamma)$ can be identified with $\Delta^\infty(\Gamma)$.  We consider the  
Wasserstein metric $W_1(\mu,\nu)$ on 
 $P_f(\Gamma)$ which   due to the Kantorovich-Rubinstein duality can be defined by the formula
$$
W_1(\mu,\nu)=\sup_{f\in Lip^1(\Gamma,\R)}|\int f d\mu-\int f d\nu|
$$
where $Lip^1(X,\R)$ denote the set of 1-Lipschitz real-valued functions on $X$. We note that the inclusion $\Gamma\subset P_f(\Gamma)$ is an isometric embedding. Then any 1-Lipschitz map $g:\Gamma\to Z$ defines a 1-Lipschitz map $P(g):P_f(\Gamma)\to P_f(Z)$. Since the barycenter map $b_\R:P_f(\R)\to\R$ defined as $b(\mu)=\int t d\mu$ is 1-Lipschitz, the composition 
$b_{\R^n}\circ P(\alpha):P_f(\Gamma)\to\R^n$ with the barycenter map  $b_{\R^n}:P_f(\R^n)\to\R^n$ is $\sqrt{n}$-Lipschitz.

The map $\bar\alpha:\Delta^\infty(\Gamma)\times P\to\R^n$ defined as  $$\bar\alpha(x,y)=\frac{1}{\sqrt{n}}b_{\R^n}\circ P(\alpha|_{\Gamma\times y})(x,y)$$
satisfies the conditions (1)-(3).
\end{ex}

\begin{prop}\label{select}
Suppose that a map $\alpha:\Gamma\times P\to\R^n$ satisfies the conditions (1)-(3) for a group $\Gamma$ with a finite CW complex $B\Gamma$ and a finite-dimensional simplicial complex $P$. Then there is a map
$\bar\alpha :E\Gamma\times P\to\R^n$, where $E\Gamma$ is the universal cover of $B\Gamma$,  satisfying the conditions 
\begin{itemize}
\item (1) $\alpha$ is invariant with respect to the diagonal action of $\Gamma$ on
$E\Gamma\times P$;
\item (2) the restriction $\alpha\mid_{y\times P}$ is proper for all $y\in E\Gamma$;
\item (3) the restriction $\alpha\mid_{E\Gamma\times p}$ is 1-Lipschitz for all $p\in P$ for the metric on $E\Gamma$ lifted from $B\Gamma$.
\end{itemize}
\end{prop}
\begin{proof}
For the proof  we use the Michael selection theorem
\begin{thm}[\cite{M}]
Let $X$ be a paracompact space and $Y$ a Banach space. Let $F:X\to Y$ be a lower semi-continuous set-valued map with nonempty closed convex values.
Then there exists a continuous selection of $F$.
\end{thm}
The map $\alpha=(\alpha_j)$ is defined by its coordinate maps $\alpha_j:\Gamma\times P\to\R$. We construct an extension $\bar\alpha_j:E\Gamma\times P\to\R$ for each $j$.
We may assume that the CW complex $B\Gamma$ has one vertex. Then the 0-skeleton of $E\Gamma$ can be identified with $\Gamma$.
We construct the map $\bar\alpha$ recursively defining it on $E\Gamma^{(i)}\times P$ for the $i$-skeleton of $E\Gamma$.
The base of induction is the map $\alpha_j$.

Suppose that $\bar\alpha_j:E\Gamma^{(i)}\times P\to\R$ satisfying (1) and (3) is constructed. We extend it separately to $\bar e^{i+1}\times P$ for each $(i+1)$-dimensional cell
$e^{i+1}\subset E\Gamma$. Let $Lip^1(X,\R)$ denote the set of 1-Lipschitz functions. 
 Let $(X,A)$ be a compact pair  and let $f_0:A\to\R$ be a 1-Lipschitz function.

 Claim: Let $(X,A)$ be a compact pair  and let $f_0:A\to\R$ be a 1-Lipschitz map. 

(i) {\em Then the set $Lip_{A,f_0}^1(X,\R)$ of 1-Lipschitz functions $g:X\to\R$ that restrict to $f_0$ on $A$,
$g|_A=f_0$, is a nonempty compact  convex set in the Banach space $C(X,\R)$.}

(ii) {\em If a map $f: P\to Lip^1(A,\R)$ is continuous, then the multi-valued  map $F:P\to C(X,\R)$ defined as $F(p)=Lip^1_{A,f(p)}(X,\R)$
is a lower semi-continuous.}

{\em Proof of the Claim:} Since a 1-Lipschitz map to $\R$ can be extended to a 1-Lipschitz map for any pair $(X,A)$~\cite{CGM}, the set is nonempty.
Note that the straight-line homotopy in $\R$ between 1-Lipschitz maps as a path that consists of 1-Lipschitz maps. This proves convexity (i). Compactness follows 
from the Ascoli-Arzela theorem.

Let $p_k\to p$ be a sequence in $P$ converging to $p$. Let $g_k\in F(p_k)$, $g_k:X\to\R$.
Going to a subsequence in view of Ascoli-Arzela theorem we may assume that $g_k$ converges to a function $g:X\to\R$. Note that $g$ is 1-Lipschitz and  $g|_A=f(p)$.
Hence, $g\in F(p)$. This implies the semi-continuity of $F$.

We apply the claim (ii) with $X=\bar e^{i+1}$, $A=\partial e^{i+1}$, and $f: P\to Lip^1(A,\R)$ obtained from $\bar\alpha_j:A\times P\to\R$ to obtain a lower semi-continuous multi-valued function
 $F:P\to C(X,\R)$ defined as $F(p)=Lip^1_{A,f(p)}(X,\R)$. By the Michael selection theorem there is a continuous selection $\phi:P\to C(X,\R)$. The map $\phi$ defines the desired extension
$\bar\alpha$ to $\bar e^{i+1}\times P$.

We do this for all $(i+1)$-dimensional cells to get $\bar\alpha_j$ defined on $E\Gamma^{(i+1)}\times P$. Since we assume that the metric on $E\Gamma^{(i+1)}$ is geodesic, the local 1-Lipschitz condition extends to global. Thus, the condition (3) is satisfied.

As the result we obtain an extension $\bar\alpha_j:E\Gamma\times P\to\R$ of $\alpha_j$ that satisfies (1) and (3). Define $\bar\alpha=(\bar\alpha_j)_{j=1}^n: E\Gamma\times P\to\R^n$.
We rescale $\R^n$ to have the condition (3).

We need to check the condition (2), i.e. to show that $\bar\alpha$ is proper on $y\times P$ for each $y\in E\Gamma$. Suppose that  $|\bar\alpha(y\times p_n)|\le R$.
Then $$|(\alpha(\gamma\times p_n)|\le|\bar\alpha(y\times p_n)|  +|(\alpha(\gamma\times p_n)- \bar\alpha(y\times p_n)|\le R+ d(y,\gamma)$$ for fixed $\gamma$.
Since $\alpha|_{\gamma\times P}$ is proper, this implies that the sequence $p_n$ is bounded in $P$.
\end{proof}

Since the map $\bar\alpha$ is invariant under the diagonal action of $\Gamma$, there is a map  $\hat\alpha:E\Gamma\times_\Gamma P\to \R^n$ such that
$\bar\alpha=\hat\alpha\bar p$ where $\bar p:E\Gamma\times P\to E\Gamma\times_\Gamma P$ is the projection to the orbit space.
Then $\bar\alpha^*(\omega)=\bar p^*\hat\alpha^*(\omega)$. In the case of finite $B\Gamma$ the map $\hat\alpha$ is proper. Therefore,
$$[\bar\alpha^*(\omega)]=\bar p^*(\hat\alpha^*([\omega]))$$ where $\hat\alpha^*([\omega])\in H_c^n(E\Gamma\times_\Gamma P)$.
Thus, the cohomology class $[\bar\alpha^*(\omega)]$ is a $\Gamma$-invariant version of $\hat\alpha^*([\omega])$.

\subsection{Manifolds with $K\le 0$} Let $M$ be a closed  parallelizable Rimannian $n$-manifold with nonpositive sectional curvature and let $\Gamma=\pi_1(M)$. Then $M$ is $B\Gamma$.
Let $q:TM\to \R^n$ be a trivialization map of the tangent bundle which is an isomorphism on each fiber. By the Cartan-Hadamard theorem the exponential map
$exp_x:T_xM\to M$ lifts to  a diffeomorphism $\widetilde{exp}_x=exp_y:T_xM\to \Wi M$, where $p:\Wi M\to M$ is the universal covering amd $y\in  p^{-1}(x)$.
The collection of maps $\widetilde{exp}_x$, $x\in M$ defines a fiberwise diffeomorphism $$\widetilde{exp}:TM\to \Wi M\times_\Gamma \Wi M.$$
We note that the inverse map $\widetilde{exp}_x^{-1}$ is 1-Lipschitz. Let $\Lambda:\Wi M\times_\Gamma \Wi M\to TM$
be the inverse for $\widetilde{exp}$ . 
Here the projection of $TM\to M$ corresponds to the projection $\Wi M\times_\Gamma\Wi M\to M$ onto the second factor.
We treat the first factor as $E\Gamma$.
Consider the map
$$\bar\alpha=q\circ\Lambda\circ\bar p:E\Gamma\times\Wi M\to\R^n.$$
The map $\hat\alpha=q\Lambda$ is proper, since $q$ is proper and $\Lambda$ is a homeomorphism. Since the restriction $\bar p|_{x\times\Wi M}:x\times\Wi M\to E\Gamma_\Gamma\Wi M$, it follows that the restriction $\bar\alpha|_{x\times\Wi M}$ is proper for all $x\in E\Gamma$. Since $\widetilde{exp_x}^{-1}$ is 1-Lipschitz, the restriction $\bar\alpha|_{E\Gamma\times y}$ is 1-Lipschitz for all $y$. Thus, the conditions (1)-(3) are satisfied. 

We note that
$q^*([\omega])$ is the Thom class $u$ of the tangent bundle $TM\to M$ which is supported on a tubular neighborhood of the zero section of $TM$. Then  
$\hat\alpha^*([\omega])=\Lambda^*(u)$ is the Thom class of the bundle $E\Gamma\times_\Gamma\Wi M\to M$ which is supported on a tubular neighborhood $V$ of the "zero" section i.e. the image
$\bar p(\Delta\Wi M)$ of the diagonal. We may assume that the neighborhood $V$ is pulled back from the tubular neighborhood of the diagonal $\Delta M\subset M\times M$
under the natural projection $p':\Wi M\times_\Gamma\Wi M\to M\times M$. Thus, $\hat\alpha^*([\omega])$ is the image under $(p')^*$
of the Thom class $u'$ of normal bundle of the diagonal in $M\times M$.

We recall that the Thom class $u'$ of the normal bundle of the diagonal $\Delta M\subset M\times M$ generates the diagonal class $u''\in H^n(M\times M)$ and
the slant product with $u''$ is the Poincare Duality isomorphism $H_k(M)\to H^{n-k}(M)$. The proof of this fact can be derived from the fomula for $u''$ given in~\cite{MS}. 
This isomorphism can be stated in terms of  $\Gamma$-equivariant homology and cohomology of $\Wi M$  as $$(\bar p p')^*(u'')/ :H^\Gamma_k(\Wi M)\to H_{n-k}^\Gamma(\Wi M).$$
Thus, the slant product $(\bar p p')^*(u')/ :H^\Gamma_k(\Wi M)\to H_{n-k}^\Gamma(\Wi M)$ is an isomorphism. Since $(p')^*(u')=\hat\alpha^*([\omega])$,
the slant product $\alpha_\cap$ is an isomorphism.

Therefore, all cohomology of $\Gamma$ are proper Lipschitz~\cite{CGM}.

If $M$ is not parallelizable the same result can be achieved by adding to $TM$ the complementary vector bundle and make the above argument with a large parameter space $P$.
Then one should use the Poincare Duality for the infinite chain homology (sometimes called the Borel-Moore homology).

\subsection{The case of aspherical manifolds}

Now we consider the case when $B\Gamma=M$ is a closed orientable aspherical $n$-manifold and
$E\Gamma=\Wi M$ is its universal cover, $\Gamma=\pi_1(M)$. Suppose that a map $\alpha:\Gamma\times P\to\R^n$ satisfies the conditions (1)-(3). 
Then by Proposition~\ref{select} there is a map $\bar\alpha:\Wi M\times P\to\R^n$ satisfying conditions (1)-(3).
First we consider the case when $P=\Wi M$. If the map $\hat\alpha: \Wi M\times_\Gamma P\to\R^n$  induces an injective homomorphism of the $n$-dimensional cohomology with compact supports, then $\hat\alpha^*([\omega])$ is the Thom class of the normal bundle of $M=P/\Gamma\subset \Wi M\times_\Gamma P$  and the argument in the example in 2.3
implies that $\alpha_\cap$ is  an isomorphism. It means that all  cohomology classes of $\Gamma$ are proper Lipschitz.

One can expand the case when $\alpha_\cap$ is an isomorphism for aspherical manifolds as follows.
Suppose  that $\Gamma$ acts on $P=\Wi M\times\R^k$ in such a way that
the projection $pr:P\to \Wi M$ is equivariant. 
Then for rational coefficients the argument with the Thom class allows to prove the following

\begin{prop}[\cite{Dr}]\label{P}
Let $\alpha:\Wi M\times P\to\R^{n+k}$ be a map
satisfying the conditions (1)-(3). Then  $\alpha_{\cap}$ is isomorphism over $\Q$  in all dimensions if and only if
$\alpha_\cap$ is not zero.
\enditemize
\end{prop}
If the conditions of Proposition~\ref{P} are satisfied for an aspherical manifold $M$ we call the cohomology groups $H^*(M;\Q)=H^*(\pi_1(M),\Q)$ {\em canonically proper Lipschitz}.
In particular, Proposition~\ref{P} allows to extend the argument of 2.3 to not necessarily parallelizable manifolds with sectional curvature $K\le 0$. We refer to~\cite{CGM} for detailed account. Briefly, if $E$ is a $\R^k$-bundle "orthogonal" to the topological tangent bundle
$\Wi M\times_\Gamma\Wi M\cong TM$ of $M$, then  the Whitney sum admits a trivialization $q:TM\oplus E\to\R^{n+k}$. The pull-back $E^*\to \Wi M$ is a trivial
vector bundle and it is $\Gamma$-equivariant, hence, $P=\Wi M\times\R^k$ satisfies the condition of Proposition~\ref{P}. Note that the orbit space of the diagonal action on $\Wi M\times P$
is the total space of the bundle $(\Wi M\times_\Gamma\Wi M)\oplus E$ and hence  we can use the Thom class to show essentiality of $\hat\alpha$.

One of the corollaries of Proposition~\ref{P} states that if an aspherical manifold $M$ has a proper Lipschitz class in some dimension (say dimension 0) with the above $P$ then all its cohomology classes are proper Lipschitz.

The following theorem about general aspherical manifolds~\cite{Dr} is a generalization of  the case of manifolds with $K(M)\le 0$.
\begin{thm}
Suppose that the universal cover of an aspherical manifold $M$ with $\pi_1(M)=\Gamma$ admits a $\Gamma$-equivariant \v Cech acyclic Higson dominated compactification, then $M$
has canonically proper Lipschitz cohomology and, hence, the Novikov conjecture holds true for $\Gamma$. 
\end{thm}
We recall that the Higson domination condition for a compactification of $E\Gamma$ is equivalent to the nulity condition from the definition of Bestvina's $Z$-boundary~\cite{Best}.

In particular this theorem implies the Novikov conjecture for the fundamental group of an aspherical manifold $\Gamma=\pi_1(M)$  when $\Gamma$ admits an $EZ$-boundary. We refer to
~\cite{FL} for the original proof of the Novicov conjecture for groups admitting $EZ$-boundary.

\section{Berstein-Schwarz class of aspherical manifolds}

\subsection{Berstein-Schwarz class} 
Let $\Gamma$ be a discrete group, $\Z\Gamma$ denote the group ring, and let $\epsilon:\Z\Gamma\to\Z$ be the augmentation homomorphism. Then the
augmentation ideal $I(\Gamma)$ is the kernel of $\epsilon$.
The Berstein-Schwarz class $\beta_\Gamma\in H^1(\Gamma,I(\Gamma))$ is the image $\beta_\Gamma=\delta(1)$ 
of the generator $1\in H^0(\Gamma,\Z)$ under connecting homomorphism $\delta:H^0(\Gamma,\Z)\to H^1(\Gamma,I(\Gamma))$ in the coefficient long exact sequence
generated by the short exact sequence
$$
0\to I(\Gamma) \to\Z\Gamma\stackrel{\epsilon}\to\Z\to 0.
$$
We recall~\cite{Br} that the cap product $\alpha_1\alpha_2$ of two cohomology classes $\alpha_i\in H^{k_i}(\Gamma,L_i)$, $i=1,2$, with coefficients in $\Z\Gamma$-modules $L_i$ belongs to $H^{k_1+k_2}(\Gamma,L_1\otimes L_2)$ where the tensor product is taken over $\Z$.

\begin{thm}[Universality]~\cite{DR},\cite{Sc}
For any $\Z\Gamma$-module $L$ and any cohomology class $\alpha\in H^k(\Gamma,L)$ there is a homomorphism $\phi:I(\Gamma)^{\otimes k}\to L$
such that $\alpha$ is the image of $(\beta_\Gamma)^k$ under the induced homomorphism
$$
\phi^*:H^k(\Gamma, I(\Gamma)^{\otimes k})\to H^k(\Gamma,L).
$$
\end{thm}
In the case when $\Gamma$ is the fundamental group of a closed orientable aspherical manifold the group $H^k(\Gamma, I(\Gamma)^{\otimes k})$ can be computed.
\begin{thm}
For a closed orientable aspherical $n$-manifold $M$ with the fundamental group $\Gamma$ for $k<n$ the cohomology group
$H^k(M;I(\Gamma)^{\otimes k})$ is an infinite cyclic group  generated by $(\beta_\Gamma)^k$.

The cohomology group $H^n(M;I(\Gamma)^{\otimes n})$ is the group of coinvariants $(I(\Gamma)^{\otimes n})_\Gamma=I(\Gamma)^{\otimes n}\otimes_\Gamma\Z$.
\end{thm}
\begin{proof}
For any $\ell$ the short exact sequences
$$
0\to I(\Gamma)\to\Z\Gamma\to\Z\to 0
$$
defines by tensor product with $I(\Gamma)^{\otimes \ell-1}$  a short exact sequence
$$
  0\to I(\Gamma)^{\otimes \ell}\to I(\Gamma)^{\otimes \ell-1}\otimes\Z\Gamma\to I(\Gamma)^{\otimes \ell-1}\to 0.
$$
We note that the $\Gamma$-module $I(\Gamma)^{\otimes \ell-1}\otimes \Z\Gamma$ is projective~\cite{DR}.
Then the coefficient long exact sequence for homology and the fact that homology of a group with coefficients in projective module
are trivial in dimensions $>0$~\cite{Br} produce the isomorphisms for $i>0$
$$
H_{i+1}(M;I(\Gamma)^{\otimes \ell-1})\to H_i(M;I(\Gamma)^{\otimes\ell}).
$$
In particular,
$$
\Z=H_n(M;\Z)=H_{n-1}(M;I(\Gamma))=H_{n-2}(M;I(\Gamma)^{\otimes 2})=\cdots=H_1(M;I(\Gamma)^{\otimes n-1}).
$$
By  the Poincare duality with local coefficients~\cite{Br} $$H^k(M;I(\Gamma)^{\otimes k})=H_{n-k}(M;I(\Gamma)^{\otimes k})=\Z$$
for $k<n$ and $$H^n(M;I(\Gamma)^{\otimes n})=H_0(\Gamma,I(\Gamma)^{\otimes n})=(I(\Gamma)^{\otimes n})_\Gamma.$$
\end{proof}

\begin{thm}\label{short exact}
For any group $\pi$  there is an isomorphism $H_1(\pi,\Z)=I(\pi)_\pi$. For $\ell>1$ there is a short exact sequence
$$
0\to H_\ell(\pi;\Z)\to (I(\pi)^{\otimes \ell})_\pi\to I(\pi)I(\pi)^{\otimes\ell-1}\to 0.
$$
\end{thm}
\begin{proof}
The equality  $H_1(\pi,\Z)=I(\pi)_\pi$ follows from the coefficients long exact sequence
$$
0=H_1(\pi,\Z\pi)\to H_1(\pi,\Z)\to H_0(\pi, I(\pi))\to H_0(\pi,\Z\pi)\to H_0(\pi,\Z)\to 0
$$
and the fact that $H_0(\pi,\Z\pi)=(\Z\pi)_\pi\to \Z=H_0(\pi,\Z)$ is an isomorphism. For $\ell>1$ we have the following exact sequence
$$
0\to H_1(\pi,I(\pi)^{\otimes\ell-1})\to H_0(\pi, I(\pi)^{\otimes\ell})\to H_0(\pi,I(\pi)^{\otimes\ell-1}\otimes\Z\pi)\to H_0(\pi,I(\pi)^{\otimes\ell-1})\to 0
$$
which defines the short exact sequence
$$
0\to H_1(\pi,I(\pi)^{\otimes\ell-1})\to H_0(\pi, I(\pi)^{\otimes\ell})\to K_{\ell-1}\to 0$$
where
$$K_\ell=\Ker\{(I(\pi)^{\otimes\ell}\otimes\Z\pi)_\pi\stackrel{1\otimes\epsilon}\longrightarrow (I(\pi)^{\otimes\ell}\otimes\Z)_\pi\}.$$
In view of the equality $(M\otimes N)_\pi=M\otimes_\pi N$ we obtain $$K_{\ell}=Ker\{I(\pi)^{\otimes\ell}\to(I(\pi)^{\otimes\ell})_\pi\}=I(\pi)I(\pi)^{\otimes\ell}.$$
The chain of isomorphisms
$$
H_{i+1}(\pi,I(\pi)^{\otimes k-1})\to H_i(\pi;I(\pi)^{\otimes k})
$$
for $i\ge 1$ defines the equalities
$$
H_1(\pi,I(\pi)^{\otimes\ell-1})=H_2(\pi,I(\pi)^{\otimes\ell-2})=\cdots=H_{l-1}(\pi,I(\pi))=H_\ell(\pi,\Z).$$
\end{proof}
\begin{prop}
For a closed orientable aspherical manifold $M$ with the fundamental group $\Gamma$,
$$
 H^n(M;I(\Gamma)^{\otimes n})=I(\Gamma)I(\Gamma)^{\otimes n-1}\oplus\Z.
$$
\end{prop}
\begin{proof} By Theorem~\ref{short exact} there is the short exact sequence
$$0\to\Z\stackrel{i}\to H^n(M;I(\Gamma)^{\otimes n})\stackrel{j}\to I(\Gamma)I(\Gamma)^{\otimes n-1}\to 0.$$
As it follows from the commutative diagram
$$
\begin{CD}
H_n(M;\Z)@>\cap\beta_\Gamma>> H_{n-1}(M.I(\Gamma))@>\cap\beta_\Gamma >>\dots @>\cap\beta_\Gamma>> H_0(M;I(\Gamma)^{\otimes n})\\
@A[M]\cap AA @A[M]\cap AA @. @A[M]\cap AA\\
H^0(M;\Z) @>\cup\beta_\Gamma>> H^1(M;I(\Gamma))@>\cup\beta_\Gamma >>\dots @>\cup\beta_\Gamma>> H^n(M;I(\Gamma)^{\otimes n})\\
\end{CD}
$$
the homomorphism $i$ being the Poincare dual to the iterated cup product with $\beta_\Gamma$ is an iterated $\cap$-product with $\beta_\Gamma$. Hence $i:H_n(M;\Z)\to H_0(M;I(\Gamma)^{\otimes n})$
takes the fundamental class $[M]$ to $[M]\cap(\beta_\Gamma)^n$, the Poincare dual to $(\beta_\Gamma)^n$.
 By the Universality of the Berstein-Schwarz class
there is a homomorphism of $\Gamma$-modules $\phi:I(\Gamma)^{\otimes n}\to \Z$ such that
the induced homomorphism
 $$\phi^*:H^n(M;I(\Gamma)^{\otimes n})\to H^n(M,\Z)=\Z$$ takes $(\beta_\Gamma)^n$
to the generator $1\in\Z$.  Let $H^n(M;\Z)=\Z\stackrel{\psi}\to \Z=H_n(M;\Z)$ be the isomorphism that takes 1 to $[M]$. We obtain the commutative diagram 
$$
\begin{CD}
H^n(M;\Z) @<\phi^*<< H^n(M;I(\Gamma)^{\otimes n})\\
@V\psi VV @V[M]\cap VV\\
H_n(M;\Z) @> -\cap(\beta_\Gamma)^n>> H_0(M;I(\Gamma)^{\otimes n})\\
\end{CD}
$$
that defines the splitting.
\end{proof}

\section{Lipschitz cohomology with local coefficients}

\subsection{Proper Lipschitz cohomology with local coefficients} 
Let $\Gamma$ be a discrete group with finite classifying complex $B\Gamma$.
We recall that he proper Lipschitz cohomology classes of $\Gamma$ were by
the following data~\cite{CGM}:
Suppose that there is a finite-dimensional connected complex $P$ with a  proper $\Gamma$-action
and a  map $\alpha:
\Gamma\times P\to\R^n$ satisfying the conditions
\begin{itemize}
\item (1) $\alpha$ is invariant with respect to the diagonal action of $\Gamma$ on
$\Gamma\times P$;
\item (2) the restriction $\alpha\mid_{\gamma\times P}$ is proper for all $\gamma\in \Gamma$;
\item (3) the restriction $\alpha\mid_{\Gamma\times p}$ is 1-Lipschitz for all $p\in P$ for the word metric on $\Gamma$.
\end{itemize}
Connes-Gromov-Moscovici used constant real coefficients $\R$. Here we present the definition which extends to local coefficients $L$ with base ring $\Z$ as well as $\R$.

We assume that the action of $\Gamma$ on $P$ is a free simplicial action.
Let $D=C_*(P)$ be the simplicial chain complex of $P$ and let $S_m$ denote the set of all $m$-simplices in $P$.
The group of $\Gamma$-invariant
$m$-chains $D_m^\Gamma(L)=(D_m\otimes L)^\Gamma$ with coefficients in $L$ consists of 
formal sums $\sum_{\sigma\in S_m}\lambda_\sigma\sigma$ 
where coefficients $\lambda_\sigma\in L$ are satisfying the equalities $\lambda_{\gamma\sigma}=\gamma\lambda_\sigma$. 
\begin{prop}\label{tensor}
If $D$ is a finitely generated $\Z\Gamma$-module, then
$D^\Gamma(L)\cong L\otimes_\Gamma D$.
\end{prop}
\begin{proof}
Since $D$ is finitely generated, every its element  $\sum_{\sigma\in S_m}\lambda_\sigma\sigma$ can be splitted in a finite sum $\sum_i\sum_{\gamma\in\Gamma}\lambda_{\gamma\sigma_i}(\gamma\sigma_i)$.
The map $\Phi:D^\Gamma(L)\to L\otimes_\Gamma D$ defined on generators as $$\Phi(\sum_{\gamma}\lambda_{\gamma\sigma}(\gamma\sigma))=\lambda_\sigma\otimes_\Gamma\sigma$$
is well-defined homomorphism, since $$\lambda_\sigma\otimes_\Gamma\sigma=\gamma\lambda_\sigma\otimes_\Gamma\gamma\sigma=\lambda_{\gamma\sigma}\otimes_\Gamma\gamma\sigma.$$  It
is an isomorphism with the inverse $\Phi^{-1}$ that takes  $\lambda\otimes_\Gamma\sigma$ to $\sum_{\gamma}\gamma\lambda(\gamma\sigma)$. 
\end{proof}

Let $C_*$ be the normal resolution of $\Gamma$, i.e. the simplicial chain complex of the simplex $\Delta^\infty(\Gamma)$ spanned by $\Gamma$~\cite{Br}.
Suppose that $\phi: (C\otimes D)_n\to\Z$ is a $\Gamma$-invariant cocycle such that the sum
$$
\sum_{\sigma}\phi(c\otimes \sigma)$$
 is finite for any $\Gamma$-invariant infinite chain $z=\sum_{\sigma}n_\sigma\sigma$  in $D^\Gamma=D^\Gamma(\Z)$ where $\sigma$ runs over all $(n-k)$-dimensional simplices 
in $P$. Thus, this sum defines  a $\Gamma$-invariant cocycle $\bar\phi:(C\otimes D^\Gamma)_n\to\Z$. 
We  define a $\Gamma$-invariant slant product $$\bar\phi/:D^\Gamma_{n-k}(L)\to Hom_\Gamma(C_k,L)$$ by the formula
$$
(\bar\phi/z)(\Delta)=\sum_{\sigma\in S_{n-k}}\phi(\Delta\otimes \sigma)\lambda_{\sigma}
$$
where $z=\sum_{\sigma\in S_{n-k}}\lambda_{\sigma}\sigma$ is a $\Gamma$-invariant $(n-k)$-chain with coefficients in $L$ and  $\Delta=[\gamma_0,\dots,\gamma_k]$ is a $k$-simplex in $\Delta^\infty(\Gamma)$.

Next, we show that the cochain $\bar\phi/z: C_k\to L$ 
is equivariant: 
 $$(\bar\phi/z)(g\Delta)=\sum_\sigma\phi(g\Delta\otimes\sigma)\lambda_{\sigma}=
\sum_\sigma\phi(\Delta\otimes g^{-1}\sigma)gg^{-1}\lambda_{\sigma}=$$
$$\sum_\sigma  \phi(\Delta\otimes g^{-1}\sigma)g\lambda_{g^{-1}\sigma}=
g(\bar\phi/z)(\Delta).$$
Here we use the fact that $g^{-1}:S_{n-k}\to S_{n-k}$ is a bijection.
The cochain $\bar\phi/z$ is a
cocycle when $z=\sum_{\gamma\in\Gamma}\lambda_{\gamma\sigma}\gamma\sigma$ is a $\Gamma$-invariant cycle.
Indeed,
it is a cocycle  since $\bar\phi$ is a cocycle and applying the Leibnitz formula $\partial(c\otimes a)=\partial c\otimes a \pm c\otimes\partial a$ 
we obtain $\delta(\bar\phi/z)=(\delta\bar\phi)/z\pm\phi/\partial z=0$. 

Then we show that  slant product with $\bar\phi$ defines the homomorphism $$H_{n-k}(D_*^\Gamma(L))\to H^k(C_*;L)$$ which we denote as $[\bar\phi]/$.
Let $z$ be a $\Gamma$-invariant cycle as above and let $z'=\sum_{\sigma}\mu_{\sigma}\sigma'$ be an invariant $(n-k+1)$-dimensional chain.
Then $\bar\phi/(z+\partial z')=\bar\phi/z+\bar\phi/\partial z'$. Note that $\bar\phi/\partial z'=\delta(\bar\phi/z')$. Therefore, $\bar\phi/(z+\partial z')$ defines the same cohomology class as
$\bar\phi/z$.

\

Assume that we have a map $\alpha:\Gamma\times P\to\R^n$ satisfying the conditions 1-3 and let $\bar\alpha:\Delta^\infty(\Gamma)\times\Gamma\to\R$ be the extension defined in Section 2.
Let $\omega$ be a cocycle on $\R^n$ whose cohomology  class $[\omega]$ generates the group $H^n_c(\R^n;\Z)=\Z$.
We define $\phi=\bar\alpha^*(\omega)$. The  homomorphism $\phi: C\otimes  D|_n\to\Z$  is defined by the formula
$$\phi(\Delta\otimes\sigma)=\omega(\bar\alpha|\Delta\times\sigma)$$
where we treat $\bar\alpha:\Delta^k\times\sigma^{n-k}\to\R^n$ as a singular chain with the staircase triangulation of the product of two oriented simplices.
The cocycle $\phi$ is $\Gamma$-invariant, since so is the map $\alpha$. The  argument  in subsection 2.2 shows that  the sum
$$
\sum_{\sigma}\phi(c\otimes \sigma)$$
 is finite.

We denote by  $H_k(P:\Gamma;L)$ the homology group defined by infinite invariant chains with coefficients in a $\Gamma$-module $L$. 
Note that if $P/\Gamma$ is compact $H_k(P:\Gamma;L)=H_k(P/\Gamma;L)$.
As in  the case of constant coefficients~\cite{CGM} we denote the corresponding slant product $[\bar\phi]/$ homomorphism  by
$$\alpha_{\cap}: H_k( P:\Gamma;L)\to H^{n-k}(\Gamma,L).$$

Thus, the cohomology class $\alpha_\cap([\sum_\sigma\lambda_\sigma\sigma])$  for a $\Gamma$-invariant cycle $\sum_\sigma\lambda_\sigma\sigma$ is defined by the cocycle 
$$[\gamma_0,\dots,\gamma_k]\mapsto\sum_\sigma\omega(\bar\alpha|[\gamma_0,\dots, \gamma_k]\times\sigma)\lambda_\sigma
.$$ 

\begin{defin} All classes in the image $\alpha_{\cap}(H_*( P/\Gamma;L))\subset
H^*(\Gamma,L)$ are called {\it proper Lipschitz cohomology classes} of a group $\Gamma$ with coefficients in the $\Z\Gamma$-module $L$.
\end{defin}

\begin{rem}
In the case of constant coefficients $L=\R$ the above formula recovers the  Connes-Gromov-Moscovici definition of Lipschitz cohomology as it presented in section 2. 
\end{rem}

\begin{rem}
In the above construction we can use cellular chain complexes $C_*$ and $D_*$ whenever $E\Gamma$ and $P$ are given $\Gamma$-equivariant CW-complex structures.
\end{rem}

\begin{thm}\label{coefficients}
The image of a proper Lipschitz class under coefficients homomorphism is proper Lipschitz.
\end{thm}
\begin{proof}
Let $h:L\to L'$ be a homomorphism of $\Z\Gamma$-modules and let $a=[\bar\phi]/[z]\in H^k(\Gamma,L)$ be a proper Lipschitz cohomology class realized by a map $\alpha:\Gamma\times P\to\R^n$. Then $h^*(a)=[\bar\phi]/h_*([b])$. Indeed, if $z=\sum_\sigma\lambda_\sigma\sigma$, then $h_*(z)=\sum_\sigma h(\lambda_\sigma)\sigma$ and
$$(\bar\phi/h_*(z))(\Delta^k)=
\sum_\sigma \phi(\Delta^k\otimes\sigma)h(\lambda_\sigma)=h^*(\bar\phi/b)(\Delta^k)).$$
\end{proof}

Given a short exact sequence of $\Z\Gamma$-modules $0\to K\to L\stackrel{\phi}\to M\to 0$ and a chain complex of free  $\Z\Gamma$-modules $D_*$
the sequence $
0\to D_*^\Gamma(K)\to D_*^\Gamma(L)\to  D_*^\Gamma(M)\to 0
$ is exact. For finitely generated $\Z\Gamma$-modules $D_*$ it follows from Proposition~\ref{tensor}. For infinitely generated $D_*$ we leave the proof to the reader.

\begin{thm}\label{slant naturality}
Let $0\to K\to L\stackrel{\phi}\to M\to 0$ be a short exact sequence of $\Z\Gamma$-modules, $C_*$ and $D_*$ be chain complexes of free  $\Z\Gamma$-modules,
and let $a^*\in H^n(C\otimes D^\Gamma;\Z)$. Then the diagram
$$
\begin{CD}
H_{n-k}(D^\Gamma(M)) @>{a^*/}>> H^k(C,M)\\
@V\bar\partial VV @V\bar\delta VV\\
H_{n-k-1}(D^\Gamma(K)) @>{a^*/}>> H^{k+1}(C,K)\\
\end{CD}
$$
commutes up to the sign,
where $\bar\partial$ and $\bar\delta$ are connecting homomorphisms in the coefficient homology and cohomology long exact sequences.
\end{thm}
\begin{proof}
Let $a$ denote a cocycle representing $a^*$ and let $z=\sum m_\sigma\sigma$ be a cycle in $D_*^\Gamma(M)$, $m_\sigma\in M$ and $m_{\gamma\sigma}=\gamma m_\sigma$.  Then 
by the Snake Lemma for homology $\bar\partial z=\sum\ell_\sigma\partial\sigma$ with $\phi(\ell_\sigma)=m_\sigma$, $\ell_\sigma\in L$ and $\ell_{\gamma\sigma}=\gamma\ell_\sigma$. Then

$$(a/\bar\partial z)(c)=\sum_\sigma a(c\otimes\partial\sigma)\ell_\sigma=\pm\sum_\sigma a(\partial c\otimes\sigma)\ell_\sigma$$
where the second equality is due to the Leibnitz formula and the fact that $a$ is a cocycle, $a(\partial(c\otimes\sigma))=0$.
For a cochain $\psi:C_k\to M$ we denote by $\bar\psi$ a lift $\bar\psi:C_k\to L$ with respect to $\phi:L\to M$. Since $(a/z)(c')=\sum_\sigma a(c'\otimes\sigma)m_\sigma$, we can define
the lift $\overline{a/z}$ by the formula $(\overline{a/z})(c')=\sum_\sigma  a(c'\otimes\sigma)\ell_\sigma$ with $\ell_\sigma$ chosen above.

By the Snake Lemma for cohomology we obtain
$$
\bar\delta(a/z)(c)=(\overline{a/z})(\partial c)=\sum_\sigma  a(\partial c\otimes\sigma)\ell_\sigma=(a/\bar\partial z)(c).
$$
\end{proof}

\begin{cor}\label{1}
Suppose that the integral 0-dimensional cohomology class  $1\in H^0(\Gamma)$ is proper Lipschitz.
Then the Berstein-Schwarz class $\beta_\Gamma\in H^1(\Gamma,I(\Gamma)$ is proper Lipschitz.
\end{cor}
\begin{proof}
We consider the short exact sequence $$0\to I(\Gamma)\to \Z\Gamma\stackrel{\epsilon}\to\Z\to 0$$ 
defined by the augmentation.
Let $\alpha:\Gamma\times P\to\R^n$ be a map satisfying 
conditions 1-3 required for proper Lipschitz property of $1\in H^0(\Gamma)$
and let $\bar\phi$ be the cocycle from the definition of $\alpha_\cap$. Since the action of $\Gamma$ on $P$ is proper, we have $n>0$.
 Thus, $H^0(\Gamma)\ni1=\alpha_\cap([z])$ for some $[z]\in H_n(P:\Gamma)$.
By Theorem~\ref{slant naturality} we obtain $\alpha_\cap(\bar\partial[z])=\bar\delta(1)=\beta_\Gamma$.
\end{proof}
\begin{prop}\label{large n}
Suppose that a class $b\in H^k(\Gamma,L)$ is proper Lipschitz via a map $\alpha:\Gamma\times P\to \R^n$. Then $b$ is proper Lipschitz via the map $\alpha\times 1:\Gamma\times P\times\R\to \R^n\times\R=\R^{n+1}$.
\end{prop}
\begin{proof}
Here we consider the pul-back action of $\Gamma$ on $P\times R$ and we take $\bar\alpha=\alpha\times 1$.
Let $s=\sum_{m\in\Z}[m,m+1]$ be the infinite 1-chain for the natural simplicial complex structure on $\R$. Note that $s$ is a cycle.
Let $\omega_0$ be a generator of $H^1_c(\R)$ supported in $1/4$-neighborhood of $1/2$.
Suppose that $[b]=\alpha_\cap([z])$ for some cycle $z=\sum_\sigma\lambda_\sigma\sigma$. Then 
$$(b/z)(\Delta)=\sum_\sigma(\omega(\bar\alpha|\Delta\times\sigma)\lambda_\sigma=
\sum_\sigma(\omega\otimes\omega_0)((\bar\alpha\times 1)|\Delta\times\sigma\times[0,1])\lambda_\sigma=$$
$$\sum_{m\in\Z}\sum_\sigma(\omega\otimes\omega_0)((\bar\alpha\times 1)|\Delta\times\sigma\times[m,m+1])\lambda_\sigma=(b/(z\otimes s))(\Delta).$$
\end{proof}
\begin{cor}\label{2}
Suppose that the integral 0-dimensional cohomology class  $1\in H^0(\Gamma)$ is proper Lipschitz.
Then cohomology classes in $H^k(\Gamma)$ are Lipschitz for all $k$.
\begin{proof}
Given $i$, in view of Proposition~\ref{large n} we may assume that $1\in H^0(\Gamma)$ is realized via a map $\alpha:\Gamma\times P\to\R^n$ with $n>i$.
Let $1=\alpha_\cap([z_1])$ for $[z_1]\in H_n(P:\Gamma)$.
By induction on  $i$ we show that $(\beta_\Gamma)^{\otimes i}$ where $z_i=\bar\partial z_{i-1}$
for $i<n$. The base of induction, $i=1$ is Corollary~\ref{1}.

Assume that  $(\beta_\Gamma)^{\otimes i-1}=\alpha_\cap([z_{i-1}])$. We apply Theorem~\ref{slant naturality} to the short exact sequence
$$
0\to I(\Gamma)^{\otimes i}\to  I(\Gamma)^{\otimes i-1}\otimes\Z\Gamma \stackrel{1\otimes\epsilon}\to I(\Gamma)^{\otimes i-1}\to 0
$$ 
to obtain $\alpha_\cap(\bar\partial z_{i-1})=\pm\bar\delta(\beta_\Gamma^{\otimes i-1})=\beta_\Gamma^{\otimes i}$.
Then in view of the universality of $\beta_\Gamma$, by virtue of Proposition~\ref{coefficients},
all $k$-dimensional cohomology classes of $\Gamma$ are proper Lipschitz.

\end{proof}
\end{cor}

\begin{thm}\label{product1}
The cup product of proper Lipschitz cohomology classes is proper Lipschitz.
\end{thm}
\begin{proof}
Let $\bar\alpha_1:\Delta^\infty(\Gamma)\times P_1\to\R^{n_1}$ and $\bar\alpha_2:\Delta^\infty(\Gamma)\times P_1\to\R^{n_2}$ be  the maps that realize Lipschitz classes $(\alpha_1)_\cap([z_1])\in H^{k_1}(\Gamma,L)$
and  $(\alpha_2)_\cap([z'])\in H^{k_2}(\Gamma,M)$ with $z=\sum\lambda_\sigma\sigma$ and $z'=\sum\mu_{\sigma'}\sigma'$. We take the product $\Delta^\infty(\Gamma)\times \Delta^\infty(\Gamma)$
with diagonal action of $\Gamma$ for the classifying space $E\Gamma$ and cosider a CW-complex $P=P_1\times P_2$ with the diagonal $\Gamma$-action. We
define $$\bar\alpha:(\Delta^\infty(\Gamma)\times \Delta^\infty(\Gamma))\times (P_1\times P_2)\to \R^{n_1}\times\R^{n_2}=\R^{n_1+n_2}$$
as $\bar\alpha(x_1\times x_2,y_1\times y_2)=\bar\alpha_1(x_1,y_1)\times\bar\alpha_2(x_2,y_2)$. We show that $$\alpha_\cap([z\otimes z'])=(\alpha_1)_\cap([z])(\alpha_2)_\cap([z']).$$
Let $k=k_1+k_2$ and $n=n_1+n_2$. Let $\Delta^i\times\Delta^{k-i}$ be a $k$-cell in $\Delta^\infty(\Gamma)\times \Delta^\infty(\Gamma)$. The cocycle $\phi$ from the definition of $\alpha_\cap$  is given by the formula $$\phi((\Delta_1\times\Delta_2)\otimes(\sigma\times\sigma'))=(\omega_1\otimes\omega_2)(\bar\alpha|(\Delta_1\times\Delta_2)\times(\sigma\times\sigma'))=$$
$$
(\omega_1\otimes\omega_2)(\bar\alpha_1|(\Delta_1\times\sigma)\times\bar\alpha_2|(\Delta_2\times\sigma'))=
\omega_1(\bar\alpha_1|(\Delta_1\times\sigma_1)\otimes \omega_2(\bar\alpha_2|(\Delta_2\times\sigma_2)
$$
$$
=\phi_1(\Delta_1\otimes\sigma)\phi_2(\Delta_2\otimes\sigma').$$
We compare the cocycles $\bar\phi/(z_1\otimes z_2)$ and $(\bar\phi_1/z_1)\cup(\bar\phi_2/z_2)$ which are both elements of  $$Hom_\Gamma(C(\Delta^\infty(\Gamma))\otimes C(\Delta^\infty(\Gamma)),L\otimes M).$$
Since $z_1\otimes z_2=\sum_{\sigma,\sigma'}(\lambda_\sigma\otimes\mu_{\sigma'})(\sigma\times\sigma')$, we obtain
$$
(\bar\phi/z\otimes z')(\Delta_1\times\Delta_2)=\sum_{\sigma,\sigma'}\phi((\Delta_1\times\Delta_2)\otimes(\sigma\times\sigma'))\lambda_\sigma\otimes\mu_{\sigma'}.
$$
Then we evaluate the cup product cocycle using the fact that when we do it with respect to the product resolution $C_*(\Delta^\infty(\Gamma))\otimes C_*(\Delta^\infty(\Gamma))$  of $\Gamma$
the cup product cocycle coincides with the cross product cocycle treated as a $\Gamma$-equivariant homomorphism instead of $(\Gamma\times\Gamma)$-equivariant (see~\cite{Br}).  Thus,
$$((\bar\phi_1/z_1)\cup(\bar\phi_2/z_2))(\Delta_1\times\Delta_2)=
\bar\phi_1/(\sum_{\sigma}\lambda_\sigma\sigma)(\Delta_1)\otimes
 \bar\phi_2/(\sum_{\sigma'}\mu_{\sigma'}\sigma')(\Delta_2)
$$
$$
=\sum_{\sigma}\phi_1(\Delta_1\times\sigma)\lambda_{\sigma}\otimes \sum_{\sigma'}\phi_2(\Delta_2\times\sigma')\mu_{\sigma'}=\sum_{\sigma,\sigma'}\phi((\Delta_1\times\Delta_2)\otimes(\sigma\times\sigma'))\lambda_\sigma\otimes\mu_{\sigma'}.
$$
Therefore, the cocycles $\bar\phi/(z_1\otimes z_2)$ and $(\bar\phi_1/z_1)\cup(\bar\phi_2/z_2)$ coincide.
\end{proof}

\subsection{General Lipschitz cohomology} Now in the definition of Lipschitz cohomology we drop the assumption of properness of the action of $\Gamma$ on $P$.
We assume that $P$ is a simplicial complex and $\Gamma$ acts by simplicial homeomorphisms.
Still the equivariant infinite chains $D^\Gamma$ for $D=C_*(P)$ are  defined and the homomorphism $\alpha_\cap$ can be defined by the same formula.
Thus, $$
(\bar\phi/z)(\Delta)=\sum_{\sigma\in S_{n-k}}\phi(\Delta\otimes \sigma)\lambda_{\sigma}
$$
for $\phi(\Delta\otimes\sigma)=\omega(\bar\alpha|\Delta\times\sigma)$
and $\alpha_\cap=[\bar\phi]$.

The difference is in formulas for $\bar\phi/z$ in the case when the chain $z=\sum_{\gamma\in\Gamma}\lambda_{\gamma\sigma}(\gamma\sigma)$
is defined by a singhle orbit. In the proper cohomology case it is
$$
\bar\phi/z(\Delta)=\sum_{\gamma\in\Gamma}\omega(\bar\alpha|\Delta\times\gamma\sigma)\lambda_{\gamma\sigma}
$$
whereas in the general case
$$\bar\phi/z(\Delta)=\sum_{\gamma\Gamma_\sigma\in\Gamma/\Gamma_\sigma}\omega(\bar\alpha|\Delta\times\gamma\sigma)\lambda_{\gamma\sigma}$$
where $\Gamma_\sigma$ is the isotropy group of $\sigma$ and the sum is indexed by the set $\Gamma/\Gamma_\sigma$ of cosets of $\Gamma_\sigma$.

The following two examples are modifications of examples from~\cite{CGM}
\begin{ex}
The integral 0-dimensional cohomology class of $\Gamma$ is Lipschitz for any group $\Gamma$.
\end{ex}
\begin{proof}
Consider the case when $P$ is a point with a trivial action of $\Gamma$ and with the constant map $\alpha:\Gamma\times P\to\R^0$.
\end{proof}
\begin{ex}
All integral 1-dimensional cohomology classes of any finitely generated  group $\Gamma$ are Lipschitz.
\end{ex}
\begin{proof}
Let $f:\Gamma\to\Z$ be a 1-cocycle. We take $P=\R$ with the action of $\Gamma$ defined as $\gamma(x)=x+f(x)$. Define $\alpha:\Gamma\times\R\to\R$  as $\alpha(\gamma,x)=-\gamma^{-1}(x)$. The axioms 1-2 are obviously satisfied. Let $S$ be a finite generating for $\Gamma$ and let $\lambda=\max\{|f(s)\mid s\in S$.
Then $$|\alpha(\gamma,x)-\alpha(\gamma s,x)=|f(\gamma^{-1})-f((\gamma s)^{-1})|=|f(\gamma s)-f(\gamma)|=|f(s)|\le\lambda d_S(\gamma s,\gamma).$$ This implies that $\alpha$ is $\lambda$-Lipschitz.
We resale $\R$ accordingly to get $\alpha$ 1-Lipschitz. 

Suppose that the generator $\omega\in H_c^1(\R)$ has support concentrated in a small neighborhood around $1/2\in\R$.
Le $z=\sum _{m\in\Z}m$ be an infinite simplicial  0-chain with respect to the triangulation of $\R$ where $\Z$ is the 0-skeleton. Then $\alpha_\cap([z])$ is defined by the cocycle
$\bar\phi/z$ which takes value
$$
(\bar\phi/z)([e,\gamma])=\sum_m\omega(\bar\alpha|[1,\gamma]\times m)=\sum_m\omega([\alpha(e,m),\alpha(\gamma,m)])=\sum_m\omega([-m,f(\gamma)-m]).
$$
Since for $r\in\Z$ only $r$ different integer translates of the interval $[0,r]$ can hit 1/2, we obtain that $(\bar\phi/z)([e,\gamma])=f(\gamma)$.
Therefore, $\alpha_\cap([z])=[f]$.
\end{proof}
\begin{question}
Is every integral 2-dimensional cohomology class of a finitely presented group Lipschitz?
\end{question}
An affirmative answer to this question would imply the Novikov conjecture for 2-dimesional classes which is already known~\cite{Ma},\cite{HS}.
\begin{rem}
We note that in the case when the action of $\Gamma$ on a simplicial complex $P$ is not free the sequence
$
0\to D_*^\Gamma(K)\to D_*^\Gamma(L)\to  D_*^\Gamma(M)\to 0
$
 is not exact. As the result, Theorem~\ref{slant naturality} does not hold any more.
\end{rem}
The proof of the following theorem coincides with the proof of Proposition~\ref{product1}.
\begin{thm}\label{product2}
The product of two Lipschitz cohomology classes is Lipschitz.
\end{thm}

\begin{question}
For which groups $\Gamma$ the Berstein-Schwarz class is Lipschitz?
\end{question}
The main result of~\cite{CGM} and  the Universality Theorem together with Theorem~\ref{coefficients} and Theorem~\ref{product2} imply the Novikov conjecture for such groups.

\end{document}